\documentclass[12pt]{article}
\usepackage{amsmath,amsfonts,amsthm,amscd,upref,amstext}
\usepackage[dvips]{graphicx}
\usepackage{subfigure}
\usepackage{pb-diagram,pb-xy}
\usepackage[all]{xy}

\newtheorem{prop}{Proposition}[section]
\newtheorem{thm}[prop]{Theorem}
\newtheorem{cor}[prop]{Corollary}

\newtheorem{lem}[prop]{Lemma}

\newtheorem*{agra}{Acknowlegment}

\numberwithin{equation}{section}

\begin{document}

\title{Castelnuovo-Mumford Regularity of the Fiber Cone for good filtrations}
\author{P. H. Lima\thanks{Work partially supported by CNPq-Brazil 141973/2010-2 and by
Capes-Brazil.}\,\,\,\,and\,\,\,V. H. Jorge P\'erez
\thanks{Work partially supported by CNPq-Brazil - Grant
309033/2009-8, Procad-190/2007. {\it Key words}: Castelnuovo-Mumford regularity, fiber cone, associated graded ring, Rees algebra, good filtration, reduction number.}}

\date{}
\maketitle

\begin{abstract}
 In this paper we show that there is a close relationship between the invariants characterizing the homogeneous vanishing of the local cohomology of the Rees algebra and the associated graded ring for the good filtrations case. We obtain relationships between the Castelnuovo-Mumford regularity of the fiber cone, associated graded ring, Rees algebra and reduction number for the good filtrations case.
\end{abstract}

\section{Introduction}
Let $(A, \mathfrak{m})$ be a commutative Noetherian local ring and $\mathfrak{F}: \ A\supset I\supset I^{2}\supset...$ a adic-filtration. Then we have important graded algebras, namely, $R(I):=\oplus_{n\geq 0}I^{n}t^{n}$, the  associated graded ring, $G(I):=\oplus_{n\geq 0}I^{n}/I^{n+1}$ and the fiber cone, $F(I):=\oplus_{n\geq 0}I^{n}/\mathfrak{m}I^{n}=R(I)/ \mathfrak{m}R(I)$. In the papers \cite{CZ1}, \cite{CZ2} and \cite{CZ}, Cortadellas and Zarzuella, studied the depth properties of the fiber cone by using certain graded modules associated to filtration of modules. Jayanthan and Nanduri, in \cite{JN}, used some results of those articles to study the regularity of the fiber cone.

The Castelnuovo-Mumford regularity of $R(I)$ and $G(I)$ are very known (\cite{HZ}, \cite{O}, \cite{H}, \cite{T}, \cite{JU}, etc). In the paper \cite{HZ}, for example, Hoa and Zarzuela obtain many results between reduction number and $a$-invariant of good filtrations. Ooishi, in \cite{O}, proved that the regularity of $R(I)$ and $G(I)$ are equal. This formula was also discovered by Johnson and Ulrich \cite{JU}. After, in \cite{T}, Trung studied the relationships between the Rees algebra and the associated graded ring, and then, he also concluded that for any ideal in a Noetherian local ring, the regularity of $R(I)$ and $G(I)$ are equal. We show this same equality for the good filtration case.

For the $I$-adic case, Jayanthan and Nanduri \cite{JN} prove that for any ideal of analytic spread one in a Noetherian local ring, the regularity of the fiber cone is bounded by the regularity of the associated graded ring. Moreover, they obtain on certain conditions that in fact the equality holds. The goal in this paper is to give an analogous theory on regularity of the fiber cone, Rees algebra, associated graded ring and reduction number for good filtration case. We show that the regularity of the fiber cone for good filtration case behave well as in the $I$-adic case.

The paper is divided into three parts. In section 2 we introduce the basic concepts about good filtration, reduction and regularity. In the section 3 we extend the Theorem 3.1, Corollary 3.2 and Corollary 3.3 of Trung \cite{T} for the good filtration case. In section 4, we obtain relationship between the regularities of the fiber cone, the associated graded ring, Rees algebra and reduction number for the good filtration case. The results of this section generalize the results of the section 2 in \cite{JN}.

\section{Preliminaries}
A sequence $\mathfrak{F}=(I_{n})_{n\geq 0}$ of ideals of $A$ is called a \emph{filtration} of $A$ if $I_{0}=A\supset I_{1}\supseteq I_{2}\supseteq I_{3}\supseteq ...,$ $I_{1}\neq A$ and $I_{i}I_{j}\subseteq I_{i+j}$ for all $i,j\geq 0$.

Given any filtration $\mathfrak{F}$ we can construct the following two graded rings
$$
R(\mathfrak{F})=A \oplus I_{1}t \oplus I_{2}t^{2}\oplus ... \ \ , \ \ G(\mathfrak{F})= A/I_{1}\oplus I_{1}/I_{2} \oplus I_{2}/I_{3} \oplus ... .
$$
We call $R(\mathfrak{F})$ the \emph{Rees algebra} of $\mathfrak{F}$ and $G(\mathfrak{F})$ the \emph{associated graded ring} of $\mathfrak{F}$.
We also denote $G(\mathfrak{F})_{+}=\oplus_{n\geq 1}I_{n}/I_{n+1}$. If $\mathfrak{F}$ is an $I$-adic filtration, i.e, $\mathfrak{F}=(I^{n})_{n\geq 0}$ for some ideal $I$, we denote $R(\mathfrak{F})$ and $G(\mathfrak{F})$ by $R(I)$ and $G(I)$ respectively. A filtration $\mathfrak{F}$ is called \emph{Noetherian} if $R(\mathfrak{F})$ is a Noetherian ring. Noetherian filtration satisfies $\cap_{n\geq 0}I_{n}=0$. By adapting the proof of \cite[Theorem 15.7]{M}, one can prove that if $\mathfrak{F}$ is Noetherian then $\dim \hspace{0.01cm} G(\mathfrak{F})= \dim \hspace{0.01cm} A$.

For any filtration $\mathfrak{F}= (I_{n})$ and any ideal $J$ of $A$, we let $\mathfrak{F}/J$ denote the filtration $((I_{n}+J)/J)_{n}$ in the ring $A/J$. Of course that if $\mathfrak{F}$ is Noetherian then $\mathfrak{F}/J$ so is.

Let $I$ be and ideal of $A$. $\mathfrak{F}$ is called an $I$-\emph{good filtration} if $II_{i}\subseteq I_{i+1}$ for all $i\geq 0$ and $I_{n+1}=II_{n}$ for all $n\gg 0$. $\mathfrak{F}$ is called a \emph{good filtration} if it is an $I$-good filtration for some ideal $I$ of $A$. $\mathfrak{F}$ is a good filtration if and only if it is a $I_{1}$-good filtration. A good filtration is a Noetherian filtration. By \cite[Theorem III.3.1.1 and Corollary III.3.1.4]{B}, $R(\mathfrak{F})$ is a finite $R(I_{1})$-module if and only if there exists an integer $k$ such that $I_{n}\subseteq (I_{1})^{n-k}$ for all $n$ if and only if $\mathfrak{F}$ is an $I_{1}$-good filtration. If a good filtration $\mathfrak{F}$ is such that $I_{1}$ is a $\mathfrak{m}$-primary ideal, it is called a \emph{Hilbert filtration}. A good filtration $\mathfrak{F}$ is called \emph{equimultiple} if $I_{1}$ is equimultiple, i.e, $s(I_{1})=  \text{ht} \ I_{1}$ (see \cite{HZ}), where $s(I_{1})$ is the analytic spread of $I_{1}$. Hilbert filtrations are equimultiple \cite[Exercise 18.2.8]{BS}. There are a lot of examples of non-$I$-adic good filtrations and Hilbert filtrations. For instance, $( \widetilde{I^{n}} )_{n\geq 0}$ is an $I$-good filtration, where $\widetilde{I^{n}}$ is the \emph{Ratliff-Rush} ideal associated to $I^{n}$ \cite{RR}. If $A$ is Local without nilpotent elements, the filtration $(\overline{I^{n}})_{n\geq 0}$ given by the integral closures of the powers of $I$ is an $I$-good filtration if and only if $A$ is analytically unramified \cite{R1}. Moreover, if $R$ is analytically unramified containing a field and $(I^{n})^{*}$ denotes the tight closure of $I^{n}$, then $((I^{n})^{*})_{n \geq 0}$ is an $I$-good filtration. In addition, if $I$ is a $\mathfrak{m}$-primary, these filtrations are Hilbert filtrations.

A \emph{reduction} of a filtration $\mathfrak{F}$ is an ideal $J\subseteq I_{1}$ such that $JI_{n}=I_{n+1}$ for $n\gg 0$. A \emph{minimal reduction} of $\mathfrak{F}$ is a reduction of $\mathfrak{F}$ minimal with respect to containment. We also know that $J\subseteq I_{1}$ is a reduction of $\mathfrak{F}$ if and only if $R(\mathfrak{F})$ is a finite $R(J)$-module \cite[Theorem III.3.1.1]{B}. If $J$ is a reduction of the $I$-adic filtration, we simply say that $J$ is a reduction of $I$. By \cite{NR}, minimal reduction of ideals always exist. If $R(\mathfrak{F})$ is a $R(I_{1})$-module, then $J$ is a reduction of $\mathfrak{F}$ if and only if $J$ is a reduction of $I_{1}$. Thus minimal reductions of good filtration always exist. For minimal reduction $J$ of $\mathfrak{F}$ we set $r_{J}(\mathfrak{F})= \sup\{n\in \mathbb{Z} \ | \ I_{n}\neq JI_{n-1}\}.$ The \emph{reduction number} of $\mathfrak{F}$ is defined as $r(\mathfrak{F})= \min \{ r_{J}(\mathfrak{F}) \ | \ J  \text{ is minimal reduction of} \ \mathfrak{F}\}$.

Let $\mathfrak{F}$ be a Noetherian filtration. For any element $x\in I_{1}$ we let $x^{*}$ denote the image of $x$ in $G(\mathfrak{F})_{1}=I_{1}/I_{2}$ and $x^{o}$ denote the image of $x$ in $F(\mathfrak{F})_{1}=I_{1}/\mathfrak{m}I_{1}$. If $x^{*}$ is a regular element of $G(\mathfrak{F})$ then $x$ is a regular element of $A$ and $G(\mathfrak{F}/(x))\cong G(\mathfrak{F})/(x^{*})$ (see \cite[Lemma 3.4]{HZ}).

An element $x\in I_{1}$ is called \emph{superficial} for $\mathfrak{F}$ if there exists an integer $c$ such that $(I_{n+1}: \ x)\cap I_{c}=I_{n}$ for all $n\geq c$. By \cite[Remark 2.10]{HZ}, an element $x$ is superficial for $\mathfrak{F}$ if and only if $(0:_{G(\mathfrak{F})} \ x^{*})_{n}=0$ for all $n$ sufficiently large. If $\text{grade} \ I_{1} \geq 1$ and $x$ is superficial for $\mathfrak{F}$ then $x$ is a regular element of $A$. To see that, let suppose that $x$ is not a zero-divisor. Thus if $ux=0$ then $(I_{1})^{c}u\subseteq \cap_{n}((I_{n}: \ x)\cap I_{c})=\cap_{n}I_{n}=0$. Hence $u=0$.

A sequence $x_{1},...,x_{k}$ is called a \emph{superficial sequence} for $\mathfrak{F}$ if $x_{1}$ is superficial for $\mathfrak{F}$ and $x_{i}$ is superficial  for $\mathfrak{F}/(x_{1},...,x_{i-1})$ for $2\leq i \leq k$.

Let $f_{1},...,f_{r}$ be a sequence of homogeneous elements of a noetherian graded algebra $S=\oplus_{n\geq 0}S_{n}$ over a local ring $S_{0}$. It is called \emph{filter-regular sequence} of $S$ if $f_{i}\not\in \mathfrak{p}$ for all primes $\mathfrak{p}\in \text{Ass}(S/(f_{1},...,f_{i-1}))$ such that $S_{+}\not\subseteq \mathfrak{p}$, $i=1,...,r$.

Let $(A, \mathfrak{m})$ be is a local ring and $\mathfrak{F}$ a good filtration. Then $v_{1},...,v_{t}\in I_{1}$ are \emph{analytically independent} in $\mathfrak{F}$ if and only if, whenever $h\in \mathbb{N}$ and $f\in A[X_{1},...,X_{t}]$ (the ring of polynomials over $A$ in $t$ indeterminates) is a homogeneous polynomial of degree $h$ such that $f(v_{1},...,v_{t})\in I_{h}\mathfrak{m}$, then all coefficients of $f$ lie in $\mathfrak{m}$. Moreover, if $v_{1},...,v_{t}\in I_{1}$ are analytically independent in $\mathfrak{F}$, and $J=(v_{1},...,v_{t})$, then $J^{h}\cap I_{h}\mathfrak{m}=J^{h}\mathfrak{m}$ for all $h\in \mathbb{N}$.

Now let define the analytic spread of a filtration $\mathfrak{F}$. The number $s=s(\mathfrak{F})=\dim R(\mathfrak{F})/\mathfrak{m}R(\mathfrak{F})$ \cite{HZ} is said to be the \emph{analytic spread} of $\mathfrak{F}$. Thus, when $\mathfrak{F}$ is the $I$-adic filtration, the analytic spread $s(I)$ equals $s(\mathfrak{F})$. Rees introduced the notion of basic reductions of Noetherian filtrations and he showed in \cite[Theorem 6.12]{R2} that $s(\mathfrak{F})$ equals the minimal number of generators of any minimal reduction. By \cite[Lemma 2.7]{HZ}, $s(\mathfrak{F})=\dim G(\mathfrak{F})/\mathfrak{m}G(\mathfrak{F})$ and by \cite[Lemma 2.8]{HZ}, $s(\mathfrak{F})=s(I_{1})$.

Let $S=\bigoplus_{n\geq 0}S_{n}$ be a finitely generated standard graded ring over a Noetherian commutative ring $S_{0}$. For any graded $S$-module $M$ we denote by $M_{n}$ the homogeneous part of degre $\emph{n}$ of $M$ and we define
$$
a(M):= \left\{ \begin{array}{ll}
           \max \{n \ | \ M_{n}\neq 0 \} & \text{if}  \  M \neq 0 \\
           -\infty & \text{if} \ M = 0
         \end{array}
         \right
 .$$
Let $S_{+}$ be the ideal generated by the homogeneous elements of positive degree of $S$. For $i\geq 0$, set $$a_{i}(S):=a(H_{S_{+}}^{i}(S)),$$ where $H_{S_{+}}^{i}(.)$ denotes the \emph{i}-th local cohomology functor with respect to the ideal $S_{+}$. More generally, for $i\geq 0$ and any graded $S$-module $M$, set
$$a_{i}(M):=a(H_{S_{+}}^{i}(M)),$$ where $H_{S_{+}}^{i}(M)$ denotes the \emph{i}-th local cohomology module of M with respect to the irrelevant ideal $S_{+}$. The \emph{Castelnuovo-Mumford regularity} (or simply \emph{regularity}) of $M$ is defined as the number $$\text{reg}(M):= \max\{a_{i}(M)+i \ | \ i\geq 0\}.$$
When $M=S$, the regularity $\text{reg} \ S$ is an important invariant of the graded ring $S$, \cite{EG} and \cite{O2}.

\section{Regularity of the Rees Algebra and the Associated Graded Ring for good filtrations}
In the paper \cite{T}, Trung show that there is a close relationship between the invariants characterizing the homogeneous vanishing of the local cohomology of the Rees algebra and the associated graded ring of an ideal. Oishi prove in \cite{O} that the Castelnuovo-Mumford regularity of the associated graded ring and the associated graded ring of an ideal are equal, i.e, $\text{reg} \ G(I)= \text{reg} \ R(I)$. Johnson and Ulrich \cite{JU} rediscovered this equality. Afterwards, Trung \cite{T} also showed this equality in a Corollary. In this section we verify generalizations of these results for good filtration case.

Let $\mathfrak{F}$ a good filtration over a ring $A$. We consider the ring $A$ as a graded ring concentrated in degree zero. Now, consider the following exact sequences
\begin{equation}\label{exac1}
0\rightarrow R(\mathfrak{F})_{+} \rightarrow R(\mathfrak{F}) \rightarrow A \rightarrow 0 ;
\end{equation}\vspace{-0.4cm}
\begin{equation}\label{exac2}
0\rightarrow R(\mathfrak{F})_{+}(1) \rightarrow R(\mathfrak{F}) \rightarrow G(\mathfrak{F}) \rightarrow 0.
\end{equation}

\begin{lem}\label{lema1}
Let $\mathfrak{F}$ a filtration over a ring $A$. We have
$$H_{R(\mathfrak{F})_{+}}^{i}(R(\mathfrak{F}))_{n}=0 \ para \ n\geq \max \{0, \ a_{i}(G(\mathfrak{F}))+1\}\ if\ i=0,1 $$
$$and \ for \ n\geq a_{i}(G(\mathfrak{F}))+1 \ if \ i\geq 2. $$
\end{lem}

\begin{proof}
We denote $H^{i}(.)=H_{R(\mathfrak{F})_{+}}^{i}(.)$. As $H^{0}(A)=A$ and $H^{i}(A)=0$ for $i\geq1$, from the exact sequence \eqref{exac1} we have \begin{equation}\label{iso1} H^{i}(R(\mathfrak{F})_{+})_{n}\simeq H^{i}(R(\mathfrak{F}))_{n} \ \text{for} \ n=0, \ i\geq2,\ \text{and} \ \text{for} \ n\neq 0, \ i\geq 0.\end{equation} As $H_{G(\mathfrak{F})_{+}}^{i}(G(\mathfrak{F}))=H^{i}(G(\mathfrak{F}))$, the exact sequence \eqref{exac2} induces the long exact sequence
$$
H^{i}(R(\mathfrak{F})_{+})_{n+1} \rightarrow H^{i}(R(\mathfrak{F}))_{n}\rightarrow H^{i}(G(\mathfrak{F}))_{n}\rightarrow H^{i+1}(R(\mathfrak{F})_{+})_{n+1}.
$$
Then we have a surjective map
$$
\begin{array}{c}
  H^{i}(R(\mathfrak{F})_{+})_{n+1} \rightarrow H_{i}(R(\mathfrak{F}))_{n} \ \text{for} \ n\geq \max \{0,a_{i}(G(\mathfrak{F}))+1\} \ \text{if} \ i=0,1 \vspace{0.3cm}\\
  \text{and} \ \text{for} \ n\geq a_{i}(G(\mathfrak{F}))+1 \ \text{if} \ i\geq2.
\end{array}
$$
By the fact that $H^{i}(R(\mathfrak{F}))_{n}=0$ for $n\gg 0$, we can conclude

\begin{equation}
\begin{array}{c}
  H^{i}(R(\mathfrak{F}))_{n}=0 \ \text{for} \ n\geq \max \{0,a_{i}(G(\mathfrak{F}))+1\} \ \text{if} \ i=0,1 \vspace{0.3cm}\\
  \text{and} \ \text{for} \ n\geq a_{i}(G(\mathfrak{F}))+1 \ \text{if} \ i\geq2.
\end{array}
\end{equation}

\end{proof}

\begin{thm}\label{teo3.1}
Let $\mathfrak{F}$ a filtration over a ring $A$. Then
\begin{itemize}
\item [\emph{(i)}] $a_{i}(R(\mathfrak{F})) \leq a_{i}(G(\mathfrak{F})), \ i\neq 1;$
\item [\emph{(ii)}] $a_{i}(R(\mathfrak{F})) = a_{i}(G(\mathfrak{F})) \  if \ a_{i}(G(\mathfrak{F})) \geq a_{i+1}(G(\mathfrak{F})), \ i\neq 1;$
\item [\emph{(iii)}] $The \ statements \ \emph{(i)} \ and \ \emph{(ii)} \ are \ true \ for \ i=1 \ if \ H_{G(\mathfrak{F})_{+}}^{1}(G(\mathfrak{F})) \neq 0 \ or \ if \ I_{1}\subseteq \sqrt{0} ;$
\item [\emph{(iv)}] $a_{1}(R(\mathfrak{F}))=-1 \ if \ H_{G(\mathfrak{F})_{+}}^{1}(G(\mathfrak{F})) = 0 \  and \ I_{1} \not\subseteq \sqrt{0} ;$
\end{itemize}
\end{thm}

\begin{proof}
We denote $H^{i}(.)=H_{R(\mathfrak{F})_{+}}^{i}(.)$.
From Lemma above, $a_{i}(R(\mathfrak{F})) \leq a_{i}(G(\mathfrak{F}))$ for $i\geq 2.$ For $i=0$, we have two cases. If $H^{0}(G(\mathfrak{F}))=0$, $a_{0}(G(\mathfrak{F}))=-\infty$. Hence, by using the Lemma \ref{lema1}, $H^{0}(R(\mathfrak{F}))_{n}=0$ para $n\geq0$. But $H^{0}(R(\mathfrak{F}))\subseteq R(\mathfrak{F})$ is a ideal, then $H^{0}(R(\mathfrak{F}))=0$. It follow that $a_{0}(R(\mathfrak{F}))=a_{0}(G(\mathfrak{F}))=-\infty$. If $H^{0}(G(\mathfrak{F}))_{n} \neq 0$, since $H^{0}(G(\mathfrak{F}))\subseteq G(\mathfrak{F})$ is a ideal, $a_{0}(G(\mathfrak{F}))\geq 0$. Hence by the Lemma \ref{lema1}, $$H^{0}(R(\mathfrak{F}))_{n}=0 \ \text{for} \ n \geq a_{0}(G(\mathfrak{F}))+1.$$ It implies $a_{0}(R(\mathfrak{F})) \leq a_{0}(G(\mathfrak{F}))$. (i) is proved.

Let show (ii): By using (i), it is enough to show $a_{i}(R(\mathfrak{F})) \geq a_{i}(G(\mathfrak{F}))$.
Clearly we can suppose $a_{i}(G(\mathfrak{F})) \neq -\infty$. If $i=0$, $H^{0}(G(\mathfrak{F})) \neq 0$ and then $a_{0}(G(\mathfrak{F}))\geq 0$. Due to \eqref{exac2} and \eqref{iso1}, $$H^{0}(R(\mathfrak{F}))_{a_{0}(G(\mathfrak{F}))} \rightarrow H^{0}(G(\mathfrak{F}))_{a_{0}(G(\mathfrak{F}))} \rightarrow H^{1}(R(\mathfrak{F}))_{a_{0}(G(\mathfrak{F}))+1}.$$ By the Lemma \ref{lema1}, we have either $a_{1}(G(\mathfrak{F}))\leq -1$ or $a_{1}(R(\mathfrak{F})) \leq a_{1}(G(\mathfrak{F})).$ In the first case, since $a_{0}(G(\mathfrak{F})) \geq 0 > -1 \geq a_{1}(R(\mathfrak{F}))$, $H^{1}(R(\mathfrak{F}))_{a_{0}(G(\mathfrak{F}))+1}=0$. Then $$H^{0}(R(\mathfrak{F}))_{a_{0}(G(\mathfrak{F}))} \rightarrow H^{0}(G(\mathfrak{F}))_{a_{0}(G(\mathfrak{F}))} \rightarrow 0.$$ Since $H^{0}(G(\mathfrak{F}))_{a_{0}(G(\mathfrak{F}))} \neq 0$, it follows that $H^{0}(R(\mathfrak{F}))_{a_{0}(G(\mathfrak{F}))}\neq 0$, i.e, $a_{0}(R(\mathfrak{F})) \geq a_{0}(G(\mathfrak{F})).$ For the second case by using the hypothesis $a_{0}(G(\mathfrak{F})) \geq a_{1}(G(\mathfrak{F}))$ we have $a_{1}(R(\mathfrak{F})) \leq a_{0}(G(\mathfrak{F}))$ so that $H^{1}(R(\mathfrak{F}))_{a_{0}(G(\mathfrak{F}))+1}=0$. It follows similarly to first case that $a_{0}(R(\mathfrak{F})) \geq a_{0}(G(\mathfrak{F}))$. Now if $i\geq 1$, by (i), $a_{i+1}(G(\mathfrak{F})) \geq a_{i+1}(R(\mathfrak{F})) $. From hypothesis, $a_{i+1}(R(\mathfrak{F}))\leq a_{i}(G(\mathfrak{F}))$. Then $$H^{i+1}(R(\mathfrak{F})_{+})_{a_{i}(G(\mathfrak{F}))+1}\simeq H^{i+1}(R(\mathfrak{F}))_{a_{i}(G(\mathfrak{F}))+1}=0.$$ From the exact sequence \eqref{exac2}, we have a surjective map $$H^{i}(R(\mathfrak{F}))_{a_{i}(G(\mathfrak{F}))}\rightarrow H^{i}(G(\mathfrak{F}))_{a_{i}(G(\mathfrak{F}))}.$$ Since $H^{i}(G(\mathfrak{F}))_{a_{i}(G(\mathfrak{F}))}\neq 0$, $H^{i}(R(\mathfrak{F}))_{a_{i}(G(\mathfrak{F}))} \neq 0$, i.e,  $a_{i}(R(\mathfrak{F}))\geq a_{i}(G(\mathfrak{F}))$.

Now let prove (iii). If $H_{G(\mathfrak{F})_{+}}^{1}(G(\mathfrak{F}))\neq 0$. By \cite[Corollary 2.3(iii)]{T} $a_{1 }(G(\mathfrak{F}))+1\geq 0$. Hence by using Lemma \ref{lema1}, $H^{1}(R(\mathfrak{F}))_{n}=0$ for $n \geq a_{1}(G(\mathfrak{F}))+1$. Then $a_{1}(R(\mathfrak{F}))\leq a_{1}(G(\mathfrak{F}))$.

If $I_{1}\subseteq \sqrt{0}$, $R(\mathfrak{F})_{+} \subseteq \sqrt{0}$ and $G(\mathfrak{F})_{+}\subseteq \sqrt{0}$. We have $\text{ara}(R(\mathfrak{F})_{+})=\text{ara}(\sqrt{R(\mathfrak{F})_{+}})=\text{ara}(\sqrt{0})=0$. Similarly, $\text{ara}(G(\mathfrak{F})_{+})=0$. By \cite[Corollary 3.3.3]{BS},
$$
\begin{array}{c}
  H^{i}(R(\mathfrak{F}))=0 \ \text{for} \ i>\text{ara}(R(\mathfrak{F}))=0 \vspace{0.1cm}\\
  \text{and} \vspace{0.1cm}\\
  H^{i}(G(\mathfrak{F}))=0 \ \text{for} \ i>\text{ara}(G(\mathfrak{F}))=0.
\end{array}
$$
In particular $a_{1}(R(\mathfrak{F}))=a_{1}(G(\mathfrak{F}))=-\infty .$

Now, let prove (iv). By the hypothesis, $H^{1}(G(\mathfrak{F}))=0$ and $I_{1}\not\subseteq \sqrt{0}$. Then $a_{1}(G(\mathfrak{F}))=-\infty$. By the Lemma \ref{lema1}, $a_{1}(R(\mathfrak{F}))\leq -1$. Let suppose that $a_{1}(R(\mathfrak{F}))<-1$. Then $H^{1}(R(\mathfrak{F}))_{-1}=0$. As $H^{0}(G(\mathfrak{F}))_{-1}=0$ since $H^{0}(G(\mathfrak{F}))\subseteq G(\mathfrak{F})$. It follow that $$H^{0}(G(\mathfrak{F}))_{-1}\rightarrow H^{1}(R(\mathfrak{F})_{+})_{0}\rightarrow H^{1}(R(\mathfrak{F}))_{-1}.$$ Then $H^{1}(R(\mathfrak{F})_{+})_{0}=0$. From the exact sequence \eqref{exac1}, $$H^{0}(R(\mathfrak{F})_{+})_{0}\rightarrow H^{0}(R(\mathfrak{F}))_{0}\rightarrow H^{0}(A) \rightarrow 0.$$ But $H^{0}(R(\mathfrak{F})_{+})\subseteq R(\mathfrak{F})_{+}$ and $(R(\mathfrak{F})_{+})_{0}=0$. Then $H^{0}(R(\mathfrak{F})_{+})_{0}=0.$ It is easy to show $H^{0}(R(\mathfrak{F}))_{0}=H_{I_{1}}^{0}(A)$. We also know that $H^{0}(A)=A$. From exact sequence above $H_{I_{1}}^{0}(A)=A$ and this implies that $I_{1}^{n}=0$ for some $n\geq 1$. By hypothesis it is a contradiction. Therefore $a_{1}(R(\mathfrak{F}))=-1$.
\end{proof}

\begin{cor}

Define $\ell:=\max\{i:H_{G(\mathfrak{F})_{+}}^{i}(G(\mathfrak{F}))\neq 0\}$. Then
\begin{itemize}

\item [\emph{(i)}] $a_{\ell}(R(\mathfrak{F}))=a_{\ell}(G(\mathfrak{F}))$
\item [\emph{(ii)}] $\ell=\max\{i:H_{R(\mathfrak{F})_{+}}^{i}(R(\mathfrak{F}))\neq 0\}$ if $I_{1}\subseteq \sqrt{0}$ or $\ell\geq 1$.

\end{itemize}

\end{cor}

\begin{proof}
For $i \geq \ell$ we have $a_{i+1}(G(\mathfrak{F}))=-\infty$. Thus we always have $a_{i}(G(\mathfrak{F}))\geq a_{i+1}(G(\mathfrak{F}))$.

If $i\neq 1$, by the Theorem \ref{teo3.1}, $a_{i}(R(\mathfrak{F}))=a_{i}(G(\mathfrak{F}))$. In this way, if $\ell>1$, $a_{\ell}(R(\mathfrak{F}))=a_{\ell}(G(\mathfrak{F}))$. If $H_{R(\mathfrak{F})_{+}}^{\ell}(R(\mathfrak{F}))=0$, $a_{\ell}(R(\mathfrak{F}))=-\infty$. Then $a_{\ell}(G(\mathfrak{F}))=-\infty.$ A contradiction. Thus $$H_{R(\mathfrak{F})_{+}}^{\ell}(R(\mathfrak{F}))\neq 0.$$ If there exist $j>\ell$ such that $H_{R(\mathfrak{F})_{+}}^{j}(R(\mathfrak{F}))\neq 0$ ($a_{j}(R(\mathfrak{F}))\neq -\infty$), by the Theorem \ref{teo3.1}(ii), $a_{j}(G(\mathfrak{F}))=a_{j}(R(\mathfrak{F}))$. A contradiction, since $H_{G(\mathfrak{F})_{+}}^{j}(G(\mathfrak{F}))=0$. Therefore if $\ell > 1$, $$\ell=\max\{i:H_{R(\mathfrak{F})_{+}}^{i}(R(\mathfrak{F}))\neq 0\}.$$

If $\ell=1$, $H_{G(\mathfrak{F})_{+}}^{1}(G(\mathfrak{F}))\neq 0$. By the Theorem \ref{teo3.1}(iii), $a_{1}(R(\mathfrak{F}))=a_{1}(G(\mathfrak{F}))$. If $H_{R(\mathfrak{F})_{+}}^{1}(R(\mathfrak{F}))=0$ then $a_{1}(R(\mathfrak{F}))=-\infty$. It implies $H_{G(\mathfrak{F})_{+}}^{1}(G(\mathfrak{F}))=0$. A contradiction. Thus $H_{R(\mathfrak{F})_{+}}^{1}(R(\mathfrak{F}))\neq 0$. If there exist $j>1$ such that $H_{R(\mathfrak{F})_{+}}^{j}(R(\mathfrak{F}))\neq 0$, by the Theorem \ref{teo3.1}(ii),
$$
a_{j}(R(\mathfrak{F}))=a_{j}(G(\mathfrak{F}))\neq -\infty.
$$
A contradiction.

Now let prove the last case. Let suppose $\ell=0$ and $I_{1}\subseteq \sqrt{0}$. From the Theorem \ref{teo3.1}(ii), $a_{0}(R(\mathfrak{F}))=a_{0}(G(\mathfrak{F}))$. If $H_{R(\mathfrak{F})_{+}}^{0}(R(\mathfrak{F}))=0$ then $a_{0}(R(\mathfrak{F}))=-\infty$ and it implies that $H_{G(\mathfrak{F})_{+}}^{0}(G(\mathfrak{F}))=0$. A contradiction. The rest is similar.
\end{proof}

\begin{cor}\label{cor}
Let $A$ a Noetherian local ring and $\mathfrak{F}$ a good filtration. Then $$\emph{reg} \ R(\mathfrak{F})=\emph{reg} \ G(\mathfrak{F}).$$
\end{cor}
\begin{proof}
By the Theorem \ref{teo3.1}(i), $a_{i}(R(\mathfrak{F}))+i \leq a_{i}(G(\mathfrak{F}))+i$ for $i\neq 1$. By the Theorem \ref{teo3.1}(iii)(iv) either $$a_{1}(R(\mathfrak{F}))+1\leq a_{1}(G(\mathfrak{F}))+1 \ \text{or} \ a_{1}(R(\mathfrak{F}))+1=0 \leq \text{reg} \ G(\mathfrak{F}).$$ In both case we have $\text{reg} \ R(\mathfrak{F}) \leq \text{reg} \ G(\mathfrak{F})$.

For the other inequality, first remember that by \cite[Proposition 3.2]{HZ}, $\text{reg} \ G(\mathfrak{F})\geq 0$. Let i be maximal such that
$$
a_{i}(G(\mathfrak{F}))+i=\text{reg} \ G(\mathfrak{F}) \geq 0.
$$
It implies that $a_{i}(G(\mathfrak{F}))+i\geq a_{i+1}(G(\mathfrak{F}))+i+1$, i.e, $a_{i}(G(\mathfrak{F}))> a_{i+1}(G(\mathfrak{F}))$. If $H_{G(\mathfrak{F})_{+}}^{i}(G(\mathfrak{F}))=0$ then $a_{i}(G(\mathfrak{F}))=-\infty$. A contradiction since $\text{reg} \ G(\mathfrak{F})\geq 0$. Thus $$H_{G(\mathfrak{F})_{+}}^{i}(G(\mathfrak{F}))\neq 0.$$ If $i\neq 1$, by the Theorem \ref{teo3.1}(ii), $a_{i}(R(\mathfrak{F}))= a_{i}(G(\mathfrak{F}))$. If $i=1$, by Theorem \ref{teo3.1}(iii), $a_{1}(R(\mathfrak{F}))= a_{1}(G(\mathfrak{F}))$ since $H_{G(\mathfrak{F})_{+}}^{1}(G(\mathfrak{F}))\neq 0$. Thus $a_{i}(R(\mathfrak{F}))= a_{i}(G(\mathfrak{F}))$ for $i\geq 0$. Therefore
$$
\text{reg} \ G(\mathfrak{F})=a_{i}(R(\mathfrak{F}))+i \leq \text{reg} \ R(\mathfrak{F}).
$$
\end{proof}

\section{Regularity of the Fiber Cone for good filtrations}
In the previous section we have shown the relationship between the $a$-invariants and regularity of the associated graded ring $G(\mathfrak{F})$ and the Rees algebra $R(\mathfrak{F})$. It is natural asks when we have inequality or equality between the fiber cone $F(\mathfrak{F})$ and the Rees algebra $R(\mathfrak{F})$. In the article \cite{JN}, Jayanthan and Nanduri obtained, under some assumption, inequalities and equalities between the regularities of $F(I)$ and $R(I)$. The main aim this section is generalize the results of the second section of \cite{JN} for good filtration case, that is, we search upper bound for the fiber cone $F(\mathfrak{F})$. The proofs are essentially the same.

Throughout Let $A$ a Noetherian Local Ring of dimension $d>0$ with an infinite residue field $A/\mathfrak{m}$ and $\mathfrak{F}$ a good filtration. Consider the
filtrations $$\mathcal{F}: A\supset \mathfrak{m} \supset \mathfrak{m}I_{1} \supset \mathfrak{m} I_{2} \supset ... $$ $$\mathfrak{F}: A  \supset I_{1} \supset  I_{2} \supset I_{3} \supset ... $$ Throughout we assume that $\mathfrak{F}\leq\mathcal{F}$, i.e, $I_{n+1}\subseteq\mathfrak{m}I_{n}$. Hence $R(\mathcal{F})$ is a $R(\mathfrak{F})$-module finitely generated since $\mathfrak{F}$ is a reduction of $\mathcal{F}$. We have two exact sequences
\begin{equation}\label{exa1}
0\rightarrow R(\mathfrak{F})\rightarrow R(\mathcal{F})\rightarrow \mathfrak{m}G(\mathfrak{F})(-1)\rightarrow 0
\end{equation}
\begin{equation}\label{exa2}
0\rightarrow \mathfrak{m}G( \mathfrak{F})\rightarrow G(\mathfrak{F})\rightarrow F(\mathfrak{F})\rightarrow 0
\end{equation}
of $R(\mathfrak{F})$-modules finitely generated.

In this article we always consider $\ell:= s(\mathfrak{F})=\dim F(\mathfrak{F})>0$.

\begin{lem}\label{lema 2}
Let $\underline{x}=x_{1},...,x_{\ell}$ generators of a minimal reduction $\mathfrak{F}$. We know that $\sqrt{R(\mathfrak{F})_{+}}=\sqrt{(x_{1},...,x_{\ell})}$. Denote $\underline{x}^{k}=(x_{1}^{k},...,x_{\ell}^{k})$. We have $a_{\ell}(R(\mathcal{F}))-1\leq a_{\ell}(R(\mathfrak{F}))$.
\end{lem}
\begin{proof}
By \cite[Corollary 35.21]{HIO},
$$
[H_{R(\mathfrak{F})_{+}}^{\ell}(R(\mathfrak{F}))]_{n}\cong \lim_{\overrightarrow{k}} \frac{I_{\ell k+n}}{(\underline{x_{k}})I_{(\ell-1)k+n}}
$$
                                                      and
$$
\ \ \ \ \ \ [H_{R(\mathfrak{F})_{+}}^{\ell}(R(\mathcal{F}))]_{n}\cong \lim_{\overrightarrow{k}} \frac{\mathfrak{m}I_{\ell k+n-1}}{(\underline{x_{k}})\mathfrak{m}I_{(\ell-1)k+n-1}} .\vspace{0.1cm}
$$
Then $a_{\ell}(R(\mathcal{F}))-1\leq a_{\ell}(R(\mathfrak{F}))$.
\end{proof}

In \cite{HZ}, Hoa and Zarzuela showed for the good filtration case that under some assumptions the regularity of the associated graded ring is equal to reduction number. In the next lemma we see when the regularity of the fiber cone is equal to reduction number.

\begin{lem}\label{lem}
Let $(A,\mathfrak{m})$ a Noetherian local ring and $\mathfrak{F}$ a good filtration such that $s(\mathfrak{F})=1$. If $J$ is a minimal reduction of $\mathfrak{F}$ and $\emph{grade} \ I_{1}=1$ then $$\emph{reg} \ F(\mathfrak{F})=r_{J}(\mathfrak{F}).$$
\end{lem}
\begin{proof}
Since $J$ is a minimal reduction of $\mathfrak{F}$, it follow by \cite{R2}, that $s(\mathfrak{F})=\mu(J)$. But by the hypothesis $s(\mathfrak{F})=1$, then $J=(a)$. By a adaptation of \cite[Proposition 18.2.4]{BS}, $a$ is independent analytically in $\mathfrak{F}$. By \cite[18.2.3]{BS}, $F(J)$ is isomorphic to $k[x]$, where $k=A/\mathfrak{m}$ since $a$ is independent analytically in $I_{1}$. We also have $J^{i}\cap I_{i}\mathfrak{m}=J^{i}\mathfrak{m}$. Then $$F(J)\hookrightarrow F(\mathfrak{F}).$$ Now by \cite[Example 9.3.1]{V} or \cite[Theorem 2.6.1]{GP} $$F(\mathfrak{F})\simeq \bigoplus_{i=1}^{e}F(J)(-b_{i})\bigoplus_{j=1}^{f}(F(J)/a^{c_{j}}F(J))(-d_{j}).$$ Let assume $b_{1}\leq ... \leq b_{e}$ and $d_{1}\leq ... \leq d_{f}$. Hence $$H_{F(\mathfrak{F})}(x)=\frac{x^{b_{1}}+...+x^{b_{e}}+(1-x^{c_{1}})x^{d_{1}}+...+(1-x^{c_{f}})x^{d_{f}}}{1-x}.$$ Observe that
\begin{equation}
\begin{array}{c}
  r_{J}(\mathfrak{F})=r_{J}(F(\mathfrak{F}))=\max \{b_{e}, d_{f} \} \vspace{0.2cm}\\
  \text{reg} \ F(\mathfrak{F})= \max \{b_{e}, c_{j}+d_{j}-1 \}
\end{array}
\end{equation}
By the hypothesis $I_{1}$ contain a regular element. Note that since $J$ is a reduction of $I_{1}$, $a$ is also a regular element. Let $r:=r_{J}(\mathfrak{F})$. If $n\geq r$ then $I_{n}=a^{n-r}I_{r}$. Hence $$\frac{I_{n}}{\mathfrak{m}I_{n}}=\frac{a^{n-r}I_{r}}{a^{n-r}\mathfrak{m}I_{r}}.$$ Then we have the map $$\frac{I_{r}}{\mathfrak{m}I_{r}} \longrightarrow \frac{I_{n}}{\mathfrak{m}I_{n}},$$ multiplication by $a^{n-r}$. It is easy to show that the map is bijective. Thus $\mu(I_{n})=\mu(I_{r})$ for $n\geq r$. Then
$$H_{F(\mathfrak{F})}(x)  =\sum_{n\geq 0}\mu(I_{n})x^{n}=\frac{1+(\mu(I_{1})-1)x+...+(\mu(I_{r})-\mu(I_{r-1}))x^{r}}{1-x}. $$ Comparing both the expression of the Hilbert series, it follow that $c_{j}+d_{j}\leq r$ and then $d_{f}\leq r-1$. In particular $r:=r_{J}(\mathfrak{F})=b_{e}$. Therefore $\text{reg} \ F(\mathfrak{F})=r_{J}(\mathfrak{F})$.

\end{proof}

In the article \cite{JN}, Jayanthan and Nanduri showed that if the analytic spread is one, the regularity of the fiber cone $F(I)$ is bounded above by the regularity of the associated graded ring $G(I)$. The next theorem show that the result is still true for any good filtration. We denote $H_{R(\mathfrak{F})_{+}}^{i}(M)$ by $H^{i}(M)$.

\begin{thm}\label{teo4.2}
Let $(A,\mathfrak{m})$ a Noetherian local ring and $\mathfrak{F}$ a good-filtration such that $\ell:= s(\mathfrak{F})=1$. Then $ \emph{reg}\ F(\mathfrak{F})\leq \emph{reg}\ G(\mathfrak{F})$. Furthermore, if $\emph{grade} \ I_{1}=1$ we have $ \emph{reg}\ F(\mathfrak{F}) = \emph{reg}\ G(\mathfrak{F})= r(\mathfrak{F})$.
\end{thm}
\begin{proof}
Since $\ell=1$, $\sqrt{R(\mathfrak{F})_{+}}=\sqrt{(x_{1})}$. As $H_{R(\mathfrak{F})_{+}}^{i}(M)=0$ para $i>\text{ara}(R(\mathfrak{F})_{+})$ and $\text{ara}(R(\mathfrak{F})_{+})\leq 1$, it follows that $H^{i}(M)=0$ for $i>1$. From exact sequence \eqref{exa2}, we have the long exact sequence

\begin{equation}
\begin{array}{ccccccccc}
  0 & \rightarrow & H^{0}(\mathfrak{m}G(\mathfrak{F}))  & \rightarrow & H^{0}(G(\mathfrak{F})) & \rightarrow & H^{0}(F(\mathfrak{F})) &  & \vspace{0.2cm} \\
   & \rightarrow & H^{1}(\mathfrak{m}G(\mathfrak{F})) & \rightarrow & H^{1}(G(\mathfrak{F})) & \rightarrow & H^{1}(F(\mathfrak{F})) & \rightarrow & 0.
\end{array}
\end{equation}

Hence $a_{0}(\mathfrak{m}G(\mathfrak{F}))\leq a_{0}(G(\mathfrak{F}))$ e $a_{1}(F(\mathfrak{F})) \leq a_{1}(G(\mathfrak{F}))$. From exact sequence \eqref{exa1}, we have the long exact sequence
\begin{equation}
\begin{array}{ccccccccc}
 0 & \rightarrow & H^{0}(R(\mathfrak{F})) & \rightarrow & H^{0}(R(\mathcal{F})) & \rightarrow & H^{0}(\mathfrak{m}G(\mathfrak{F}))(-1) & & \vspace{0.2cm}\\
   & \rightarrow & H^{1}(R(\mathfrak{F})) & \rightarrow & H^{1}(R(\mathcal{F})) & \rightarrow & H^{1}(\mathfrak{m}G(\mathfrak{F}))(-1) & \rightarrow & 0.
\end{array}
\end{equation}

Hence $a_{1}(\mathfrak{m}G(\mathfrak{F})(-1)) = a_{1}(\mathfrak{m}G(\mathfrak{F}))+1 \leq a_{1}(R(\mathcal{F})) $ and by the Lemma \ref{lema 2} we have $$a_{1}(\mathfrak{m}G(\mathfrak{F})) \leq a_{1}(R(\mathfrak{F})).$$ By hypothesis $\ell=1$ and by a remark of \cite[p. 2818]{T}, we have $H^{1}(G(\mathfrak{F}))\neq 0$. Hence, by the Theorem \ref{teo3.1}, we have $a_{1}(R(\mathfrak{F})) \leq a_{1}(G(\mathfrak{F}))$. Then $a_{1}(\mathfrak{m}G(\mathfrak{F})) \leq a_{1}(G(\mathfrak{F}))$. Therefore $$\text{reg} \ \mathfrak{m}G(\mathfrak{F}) \leq \text{reg} \ G(\mathfrak{F}).$$ By using the exact sequence \eqref{exa2} and \cite[Exercise 15.2.15]{BS} we have $$\text{reg} \ F(\mathfrak{F})\leq \text{max}\{\text{reg} \ G(\mathfrak{F}), \text{reg}  \ \mathfrak{m}G(\mathfrak{F})-1 \}= \text{reg} \ G(\mathfrak{F}).$$

Now let suppose that $\text{grade} \ I_{1}=1$. By the Lemma \ref{lem}, we have reg $ F(\mathfrak{F})= r_{J}(\mathfrak{F})$, for any minimal reduction $J$ of $\mathfrak{F}$. Furthermore by \cite[Proposition 3.6]{HZ}, reg $G(\mathfrak{F})=r_{J}(\mathfrak{F})$. Therefore reg $F(\mathfrak{F})=$reg $ G(\mathfrak{F})=r(\mathfrak{F}).$

\end{proof}
\begin{cor}
Let $(A,\mathfrak{m})$ be a Noetherian local ring and $\mathfrak{F}$ a good filtration such that $s(\mathfrak{F}):= \dim F(\mathfrak{F})=1$. If $(A,\mathfrak{m})$ is Cohen-Macaulay and $\mathfrak{F}$ is equimultiple we have $ \emph{reg}\ F(\mathfrak{F}) = \emph{reg}\ G(\mathfrak{F})= r(\mathfrak{F})$.
\end{cor}
\begin{proof}
Since $A$ is Cohen-Macaulay and $\mathfrak{F}$ is equimultiple, by \cite[Lemma 2.8]{HZ} and by \cite[Corollary 7.7.10]{GP}, we have $s(\mathfrak{F})=s(I_{1})=h(I_{1})= \text{grade} \ I_{1}$. The result follows by the Theorem \ref{teo4.2}.

\end{proof}

\begin{lem}\label{lema 3}
Let $(A,\mathfrak{m})$ be a local ring and $a\in I_{1}$. If $a^{*}$ is a regular element of $G(\mathfrak{F})$, then $$\frac{F(\mathfrak{F})}{a^{o}F(\mathfrak{F})}\cong F(\mathfrak{F}/(a)).$$
\end{lem}
\begin{proof}
Since $a^{*}$ is a regular element of $G(\mathfrak{F})$ it is easy to show that $aA\cap I_{n}=aI_{n-1}$ for any $n\geq 0$. Note that
$$
F(\mathfrak{F}/(a))=\bigoplus_{n\geq 0}\frac{\frac{I_{n}+aA}{aA}}{\frac{\mathfrak{m}}{aA}(\frac{I_{n}+aA}{aA})}=\bigoplus_{n\geq 0}\frac{\frac{I_{n}+aA}{aA}}{\frac{\mathfrak{m}I_{n}+aA}{aA}}\simeq \bigoplus_{n\geq0}\frac{I_{n}+aA}{\mathfrak{m}I_{n}+aA}.
$$
The natural map $$I_{n}\longrightarrow \frac{I_{n}+aA}{\mathfrak{m}I_{n}+aA}$$ induces a isomorphism $$\frac{I_{n}}{\mathfrak{m}I_{n}+(I_{n}\cap aA)}\simeq\frac{I_{n}+aA}{\mathfrak{m}I_{n}+aA}.$$ Note that
$$
\left(\frac{F(\mathfrak{F})}{a^{o}F(\mathfrak{F})}\right)_{n}=\frac{F(\mathfrak{F})_{n}}{(a^{o}F(\mathfrak{F}))_{n}}=
\frac{\frac{I_{n}}{\mathfrak{m}I_{n}}}{\frac{aI_{n-1}+ \mathfrak{m}I_{n}}{\mathfrak{m}I_{n}}}\simeq \frac{I_{n}}{aI_{n-1}+\mathfrak{m}I_{n}}=\frac{I_{n}}{(aA\cap I_{n})+\mathfrak{m}I_{n}}.
$$
\end{proof}

Now it is nature to ask if the regularity of the fiber cone may be greater than the regularity of the associated graded ring and under which conditions we have the equality. In the article \cite{CZ}, Cortadellas and Zarzuela proved for the adic-filtrations case, that if the depth of $F(I)$ and $G(I)$ is at least $s(I)-1$,  their regularities are equal. In \cite{JN}, Jayanthan and Nanduri generalized this result. We verify that it remains valid for any good filtration.

\begin{thm}
Let $(A,\mathfrak{m})$ be a Noetherian local ring and $\mathfrak{F}$ a good filtration. Let suppose $\emph{grade} \ I_{1}=\ell$ and $\emph{grade} \hspace{0.04cm} G(\mathfrak{F})_{+}\geq \ell-1$. Then $\emph{reg} \hspace{0.05cm} F(\mathfrak{F})\geq \emph{reg} \hspace{0.05cm} G(\mathfrak{F})$. Futhermore, if $\emph{depth}\hspace{0.01cm} F(\mathfrak{F})\geq \ell-1$, then $$\emph{reg} \hspace{0.05cm} F(\mathfrak{F})=\emph{reg} \hspace{0.05cm} G(\mathfrak{F}).$$
\end{thm}

\begin{proof}
If $\ell=1$, by the Theorem \ref{teo4.2} the result is true. Then we can suppose $\ell \geq 2$. By \cite[Proposition 2.2]{JV}, there exist generators $x_{1},...,x_{\ell}$ of a minimal reduction $J$ of $I_{1}$ such that $x_{1}^{*},...,x_{\ell}^{*}\in I_{1}/I_{2}$ is filter-regular sequence of $G(\mathfrak{F})$ and $x_{1}^{o},...,x_{\ell}^{o}\in I_{1}/\mathfrak{m}I_{1}$ is filter-regular of $F(\mathfrak{F})$. By hypothesis $\text{grade} \ G(\mathfrak{F})_{+}\geq \ell-1$. Thus $x_{1}^{*},...,x_{\ell-1}^{*}$ is $G(\mathfrak{F})$-regular due to \cite[Lemma 2.1]{HM}. We denote $\overline{\mathfrak{F}}=\frac{\mathfrak{F}}{(x_{1},...,x_{\ell-1})}$. By \cite[Lemma 3.4]{HZ}, $$G(\overline{\mathfrak{F}})\simeq G(\mathfrak{F})/(x_{1}^{*},...,x_{\ell-1}^{*}).$$ By Lemma \ref{lema 3}, $$F(\overline{\mathfrak{F}})\simeq F(\mathfrak{F})/(x_{1}^{o},...,x_{\ell-1}^{o}).$$ Since $x_{1}^{*},...,x_{\ell-1}^{*}$ are regular, $\text{reg} \ G(\overline{\mathfrak{F}})=\text{reg} \ G(\mathfrak{F})$. By \cite[Proposition 2.5]{JV}, $\dim F(\overline{\mathfrak{F}})=\dim F(\mathfrak{F})-(\ell-1)=1$. By using \cite[Proposition 3.5]{HM} and \cite[Proposition 1.2.10(d)]{BH}, we have $$\text{grade} \ \frac{I_{1}}{(x_{1},...,x_{\ell-1})}=\text{grade} \ I_{1}-(\ell-1)=1.$$ Thus, by Theorem \ref{teo4.2} we achieve $\text{reg} \ F(\overline{\mathfrak{F}})=\text{reg} \ G(\overline{\mathfrak{F}})=\text{reg} \ G(\mathfrak{F})$. By \cite[Proposition 18.3.11]{BS}, $\text{reg} \ F(\overline{\mathfrak{F}})\leq \text{reg} \ F(\mathfrak{F})$ and it implies that $$
\text{reg} \ G(\mathfrak{F})\leq \text{reg} \ F(\mathfrak{F}).
$$

Now, let assume $\text{depth} \ F(\mathfrak{F})\geq \ell-1$. Then $x_{1}^{o},...,x_{\ell-1}^{o}$ is $F(\mathfrak{F})$-regular. Then $\text{reg} \ F(\mathfrak{F})= \text{reg} \ F(\overline{\mathfrak{F}})$. Therefore $\text{reg} \ F(\mathfrak{F})= \text{reg} \ G(\mathfrak{F})$.
\end{proof}
\begin{prop}\label{prop4.6}
Let $(A,\mathfrak{m})$ a Noetherian local ring and $\mathfrak{F}$ a good filtration. If $\emph{reg} \ R(\mathcal{F})\leq \emph{reg} \ R(\mathfrak{F})$, then $$\emph{reg} \ F(\mathfrak{F})=\emph{reg} \ G(\mathfrak{F}).$$
\end{prop}
\begin{proof}
By using the properties of regularity for exact sequences \cite[Exercise 15.2.15]{BS}, we can conclude by the exact sequence \eqref{exa1} that $$\text{reg} \ \mathfrak{m}G(\mathfrak{F})(-1)=\text{reg} \ \mathfrak{m}G(\mathfrak{F})+1\leq \max\{\text{reg} \  R(\mathfrak{F})-1, \text{reg} \ R(\mathcal{F}) \}.$$ But $\text{reg} \ R(\mathcal{F})\leq \text{reg} \ R(\mathfrak{F}),$ then from the exact sequence \eqref{exa2}
$$
\text{reg} \ F(\mathfrak{F}) \leq \max\{\text{reg} \ \mathfrak{m}G(\mathfrak{F})-1, \text{reg} \ G(\mathfrak{F}) \}\leq \text{reg} R(\mathfrak{F})
$$
since $\text{reg} \ R(\mathfrak{F})= \text{reg}\ G(\mathfrak{F})$ (Corollary \ref{cor}). Thus $\text{reg} \ R(\mathfrak{F}) \leq \text{reg} G(\mathfrak{F})$ as we required.
\end{proof}

\begin{prop}\label{prop1}
Let $(A,\mathfrak{m})$ be a Noetherian local ring such that $\emph{grade} \hspace{0.05cm} I_{1}>0$. Let suppose that $I_{n}=\mathfrak{m}I_{n-1}$ for $n\gg 0$. Then $$\emph{reg} \hspace{0.05cm} F(\mathfrak{F})= \emph{reg} \hspace{0.05cm} G(\mathfrak{F}).$$
\end{prop}
\begin{proof}
Let $n_{0}$ be such that $I_{n}=\mathfrak{m}I_{n-1}$ for $n\geq n_{0}$. Note that $R(\mathfrak{F})_{+}^{n_{0}}\mathfrak{m}G(\mathfrak{F})=0$. Since $\mathfrak{m}G(\mathfrak{F})$ is a $R(\mathfrak{F})_{+}$-torsion, $$H_{R(\mathfrak{F})_{+}}^{i}(\mathfrak{m}G(\mathfrak{F}))=0,$$ for $i>0$ and $$H_{R(\mathfrak{F})_{+}}^{0}(\mathfrak{m}G(\mathfrak{F}))=\mathfrak{m}G(\mathfrak{F}).$$ If $\mathfrak{m}G(\mathfrak{F})=0$, we have $\mathfrak{m}I_{n}=I_{n+1}$ for all $n$. Thus $F(\mathfrak{F})=G(\mathfrak{F})$ and the proposition follows trivially. Let suppose that $\mathfrak{m}G(\mathfrak{F})\neq 0$. From the exact sequence \eqref{exa2} we can conclude that $$0\rightarrow H^{0}(\mathfrak{m}G(\mathfrak{F})) \rightarrow H^{0}(G(\mathfrak{F}))\rightarrow H^{0}(F(\mathfrak{F})) \rightarrow 0$$ and $H^{i}(G(\mathfrak{F}))\cong H^{i}(F(\mathfrak{F}))$ for $i>0$. Hence $a_{0}(F(\mathfrak{F}))\leq a_{0}(G(\mathfrak{F}))$ and $a_{i}(G(\mathfrak{F}))= a_{i}(F(\mathfrak{F}))$ for $i>0$. We claim that $\text{depth} \hspace{0.05cm} G(\mathfrak{F})$ equal zero. If $\text{depth} \hspace{0.05cm} G(\mathfrak{F})>0$, similarly to \cite[Theorem 8]{P}, $H^{0}(G(\mathfrak{F}))=0$. Then, from the exact sequence above, $H^{0}(\mathfrak{m}G(\mathfrak{F}))=0$, a contradiction. By \cite[Proposition 3.5]{HZ} and by the hypothesis $$a_{0}(F(\mathfrak{F}))\leq a_{0}(G(\mathfrak{F})) < a_{1}(G(\mathfrak{F}))= a_{1}(F(\mathfrak{F})).$$ Therefore $$\text{reg} \hspace{0.05cm} F(\mathfrak{F})=\max\{a_{i}(F(\mathfrak{F}))+i:i\geq 1 \}=\max\{a_{i}(G(\mathfrak{F}))+i:i\geq 1 \}=\text{reg} \hspace{0.05cm} G(\mathfrak{F}).$$
\end{proof}

\begin{prop}\label{prop2}
Let $(A,\mathfrak{m})$ be a Noetherian local ring and $\mathfrak{F}$ a good filtration.
Let suppose that $\mathfrak{m}G(\mathfrak{F})$ is $R(\mathfrak{F})$-module Cohen-Macaulay
of dimension $\ell$. Then
\begin{itemize}
\item [\emph{(i)}] $\emph{reg} \ F(\mathfrak{F})\leq \emph{reg} \ G(\mathfrak{F})$;
\item [\emph{(ii)}] if $a_{\ell}(R(\mathcal{F}))-1<a_{\ell}(R(\mathfrak{F}))$ then $\emph{reg} \ F(\mathfrak{F}) = \emph{reg} \ R(\mathfrak{F})$;
\item [\emph{(iii)}] if $a_{\ell}(R(\mathcal{F}))-1=a_{\ell}(R(\mathfrak{F}))$ then $\emph{reg} \
\mathfrak{m}G(\mathfrak{F}) \leq \emph{reg} \ R(\mathfrak{F})$ and $\emph{reg} \ F(\mathfrak{F})\leq \emph{reg} \ R(\mathfrak{F})$.
\end{itemize}

\noindent Futhermore, if $\emph{reg} \ \mathfrak{m}G(\mathfrak{F}) < \emph{reg} \ G(\mathfrak{F})$ then $\emph{reg} \ F(\mathfrak{F})=\emph{reg}  \ R(\mathfrak{F})$.
\end{prop}
\begin{proof}

Since $\mathfrak{m}G(\mathfrak{F})$ is also Cohen-Macaulay over $G(\mathfrak{F})$, $H^{i}(\mathfrak{m}G(\mathfrak{F}))=0$ for $i\neq \ell$. From the exact sequence \eqref{exa1} we have $$...\rightarrow H^{i}(R(\mathfrak{F}))\rightarrow H^{i}(R(\mathcal{F}))\rightarrow H^{i}(\mathfrak{m}G(\mathfrak{F})(-1))\rightarrow H^{i+1}(R(\mathfrak{F}))\rightarrow ... $$ Then $H^{i}(R(\mathfrak{F}))\cong H^{i}(R(\mathcal{F}))$ for $i\neq \ell$ so that $a_{i}(R(\mathcal{F}))=a_{i}(R(\mathfrak{F}))$ for $i\neq \ell$.

First, let prove (ii). By hyphotesis we have $a_{\ell}(R(\mathcal{F}))\leq a_{\ell}(R(\mathfrak{F}))$. Therefore $a_{i}(R(\mathcal{F}))\leq a_{i}(R(\mathfrak{F}))$ for all $i$ so that $\text{reg} \ R(\mathcal{F}) \leq \text{reg} \ R(\mathfrak{F})$. From the Proposition \ref{prop4.6}, $\text{reg} \ F(\mathfrak{F})=\text{reg} \ R(\mathfrak{F})$.

Now, let prove (iii). If $a_{\ell}(R(\mathcal{F}))-1= a_{\ell}(R(\mathfrak{F}))$ we have
$$
\text{reg} \ R(\mathcal{F})\leq \ \text{reg} \ R(\mathfrak{F})+1
$$
since $a_{i}(R(\mathcal{F}))=a_{i}(R(\mathfrak{F}))$ for $i\neq \ell$. From the exact sequence \eqref{exa1} we have
$$
\begin{array}{lll}
  \text{reg} \ (\mathfrak{m}G(\mathfrak{F})(-1)) & = & \text{reg} \  \mathfrak{m}G(\mathfrak{F})+1 \vspace{0.3cm}\\
   & \leq & \max \{ \text{reg} \ R(\mathfrak{F})-1  ,\text{reg} \ R(\mathcal{F})  \} \vspace{0.3cm}\\
   & \leq & \text{reg} \ R(\mathfrak{F})+1.
\end{array}
$$
From the exact sequence \eqref{exa2},
$$
\text{reg} \ F(\mathfrak{F})\leq \max \{ \text{reg} \ \mathfrak{m}G(\mathfrak{F})-1  ,\text{reg} \ G(\mathfrak{F})  \} \leq \text{reg} \ R(\mathfrak{F}).
$$ Hence (iii) is proved. By Lemma \ref{lema 2}, $a_{\ell}(R(\mathcal{F}))-1\leq a_{\ell}(R(\mathcal{F})).$ Then by (ii) and (iii), we have (i).

Finally, let assume $\text{reg} \ \mathfrak{m}G(\mathfrak{F})< \text{reg} \ G(\mathfrak{F}).$ From the exact sequence \eqref{exa2},
\begin{equation}\label{e}
\text{reg} \ G(\mathfrak{F})\leq \max \{ \mathfrak{m}G(\mathfrak{F}) ,\text{reg} \ F(\mathfrak{F}) \}.
\end{equation}
Since $\text{reg} \ \mathfrak{m}G(\mathfrak{F})< \text{reg} \ G(\mathfrak{F})$, by \eqref{e}, $$\max \{ \mathfrak{m}G(\mathfrak{F}) ,\text{reg} \ F(\mathfrak{F}) \}=\text{reg} \ F(\mathfrak{F}).$$ Thus $\text{reg} \ G(\mathfrak{F})\leq \text{reg} \ F(\mathfrak{F})$. By Corollary \ref{cor}, $\text{reg} \ G(\mathfrak{F})=\text{reg} \ R(\mathfrak{F})$. Hence $\text{reg} \  R(\mathfrak{F})\leq \text{reg} \ F(\mathfrak{F})$. The other inequality already was proved.
\end{proof}

\begin{agra}
We would like to thank the professors J. V. Jayanthan and R. Nanduri for the clarifications.
\end{agra}

\textbf{Department of Mathematics, Institute of Mathematics and Computer Science, ICMC, University of S\~{a}o Paulo, BRAZIL.
}

\emph{E-mail address}: apoliano27@gmail.com

\hspace{1cm}

\textbf{Department of Mathematics, Institute of Mathematics and Computer Science, ICMC, University of S\~{a}o Paulo, BRAZIL.
}

\emph{E-mail address}: vhjperez@icmc.usp.br 

\hspace{1cm}
\end{document}